\documentclass[10pt, letterpaper]{amsart}
\title{Rademacher Series for $\eta$-quotients}
\author{Ethan Sussman}
\date{October 17 2017}
\email{ethanws@stanford.edu}
\address{Physics Department, Stanford University, Stanford, California}
\usepackage{import}			
\usepackage[utf8]{inputenc} 	
\usepackage[english]{babel} 
\usepackage{microtype} 			
\usepackage{geometry} 			
\usepackage{amssymb}	 		
\usepackage{bm}
\usepackage{mathrsfs}			
\usepackage{xcolor}				
\usepackage{tikz}				
\usepackage[all]{xy}				
\usepackage{graphicx}			
\usepackage{wrapfig}
\usepackage{enumitem} 			
\usepackage{lmodern}			
\usepackage[T1]{fontenc}		
\graphicspath{/Graphics} 
\theoremstyle{definition}

\newtheorem*{definition*}{Definition}

\newtheorem*{problem*}{Problem}
\newtheorem{theorem}{Theorem}[section]
\newtheorem{lemma}{Lemma}[section]
\newtheorem*{lemma*}{Lemma}

\newtheorem*{corollary*}{Corollary}
\newtheorem{proposition}{Proposition}[section]
\newtheorem*{proposition*}{Proposition}

\newtheorem*{solution*}{Solution}

\newcommand{\CC}{\mathbb{C}}
\newcommand{\DD}{\mathbb{D}}

\newcommand{\HH}{\mathbb{H}}

\newcommand{\NN}{\mathbb{N}}

\newcommand{\QQ}{\mathbb{Q}}
\newcommand{\RR}{\mathbb{R}}

\newcommand{\ZZ}{\mathbb{Z}}

\newcommand{\calF}{\mathcal{F}}

\newcommand{\calK}{\mathcal{K}}

\newcommand{\calM}{\mathcal{M}}

\newcommand{\dd}[1]{\,\mathrm{d}#1} 

\newcommand{\SL}{\textrm{SL}}

\newcommand{\lcm}{\operatorname{lcm}} 
\renewcommand{\gcd}{\operatorname{gcd}} 

\newcommand{\mat}[4]{\begin{pmatrix} #1&#2\\#3&#4 \end{pmatrix}}
\renewcommand{\Re}{\operatorname{Re}} 

\begin{document}
\maketitle

\begin{abstract}
	We apply Rademacher's method in order to compute the 
	Fourier coefficients of a large class of $\eta$-quotients. 
\end{abstract}


\section{Background}

A partition of an integer $n$ is a multiset of positive integers whose sum 
is $n$. Let $p(n)$ denote the number of partitions of $n$. The value of $p(n)$ 
may be computed by brute force for sufficiently small $n$ by simply enumerating 
all possible partitions and then counting. However, $p(n)$ grows rapidly and 
brute force computation rapidly becomes intractible. Another technique, 
pioneered by Euler \cite{euler1748introductio}, is to study the properties of 
the generating function
\begin{equation}
Z(q) = \sum_{n=0}^\infty p(n) q^n = \frac{1}{\prod_{n=1}^\infty (1-q^n)}.
\label{eq:first_eq}
\end{equation}
He showed that
\begin{equation}
\frac{1}{Z(q)} = \prod_{n=1}^\infty (1-q^n) =\sum_{k\in \ZZ} (-1)^k
q^{k(3k-1)/2}.
\end{equation}
This is a result regarding formal series. However, we can interpret these 
series as complex valued functions in some appropriate domain. Viewed as 
complex functions, $Z(q)$ and $1/Z(q)$ are both nonvanishing and holomorphic on 
the open unit disk $\DD\subset \CC$ and cannot be analytically 
continued beyond $\DD$. One naturally asks (i) what are the analytic properties 
of these generating functions, and 
(ii) what properties of $p(n)$ may be deduced from these analytic properties.
With regards to (ii), if we know sufficiently many details regarding the 
analytic properties of $Z(q)$, $p(n)$ may simply be extracted by 
performing a Fourier-Laplace transform:
\begin{equation}
p(n) = \frac{1}{2\pi i} \int_\gamma \frac{Z(q)}{q^{n+1}} \dd q 
\label{eq:first_fourier_trans}
\end{equation}
for some suitably chosen contour $\gamma$. The difficulty lies in computing 
this contour integral. 

Before continuing, a word on notation: Given a function $f(q) : \DD \to \CC$, 
we 
may pull back $f(q)$ by the map $q = e^{2\pi i \tau}$ to get a new function 
$f(e^{2\pi i \tau}) : \HH \to \CC$, where $\HH$ is the open upper-half of the 
complex plane. We will often denote $f(e^{2\pi i \tau})$ as simply $f(\tau)$ 
when no confusion should arise. In order to keep track of which variable we are 
working with, we will often refer to two copies of $\CC$ as 
the $q$-plane and the $\tau$-plane. 

The first progress with regards to (i) was due to Dedekind. Dedekind considered 
the eponymous function $\eta:\DD\to \CC$ 
\begin{equation}
\eta(q) = q^{1/24} \prod_{n=1}^\infty (1-q^n)
\label{eq:eta_def}
\end{equation}
where $q=e^{2\pi i \tau}$
\cite{Apostol}. This is 
simply $1/Z(q)$ with a mysterious additional factor of $q^{1/24}$. Dedekind 
showed that $\eta(\tau)$ is a modular 
form of weight $1/2$. That is, for any matrix 
\begin{equation}
M  = \mat{a}{b}{c}{d} \in 
\SL_2(\ZZ),
\end{equation}
where $\SL_2(\ZZ)$ is the group of integral matrices with determinant $+1$,
\begin{equation}
\eta\left( \frac{a\tau+b}{c\tau+d} \right) = \epsilon(a,b,c,d) \sqrt{c\tau+d} 
\cdot
\eta(\tau) \label{eq:eta_transform_law}
\end{equation}
where $\epsilon(a,b,c,d)=\epsilon(M):\SL_2(\ZZ)\to \CC$ is a homomorphism. This 
phase factor is 
called a multiplier system. Dedekind computed $\epsilon(M)$. It is given by 
\begin{equation}
\epsilon(a,b,c,d) = 
\begin{cases}
\exp\left( +\frac{\pi i b}{12} \right) & (c=0,d=1), \\
\exp\left( -\frac{\pi i b}{12} \right) & (c=0,d=-1), \\
\exp\left( \pi i\left[ \frac{a+d}{12c}-s(d,c)-\frac{1}{4} \right]\right) & 
(c>0), \\ 

\exp\left( \pi i\left[ \frac{a+d}{12c}-s(-d,-c)-\frac{1}{4} \right]\right) & 
(c<0). \\
\end{cases}
\label{eq:mult_system}
\end{equation}
Here
\begin{equation}
s(h,k) = \sum_{n=1}^{k-1} \frac{n}{k}\left( \frac{hn}{k} - 
\left\lfloor\frac{hn}{k}\right\rfloor - \frac{1}{2}\right)
\end{equation}
is known as a Dedekind sum. In retrospect, Eq. \ref{eq:eta_transform_law} is 
rather 
remarkable, and allows us to extract asymptotics of $\eta(\tau)$ for $\tau$ 
near a given rational $q\in \QQ\subset \CC$ in terms of the asymptotics of 
$\eta(\tau)$ near $+i\infty$, which are incredibly simple: As $\tau \to 
+ i \infty$, that is as $q\to 0$, $\eta(q)\sim q^{1/24}$. Hardy and Ramanujan 
\cite{hardy244asymptotic} followed by Rademacher 
\cite{rademacher1938partition}\cite{Rademacher43} 
used this in order to carry out 
the Fourier transform in Eq. 
\ref{eq:first_fourier_trans} and therefore compute $p(n)$. This idea is rather 
general, and can be used to compute the Fourier 
coefficients of a wide variety of automorphic forms \cite{RademacherOriginal}. 
Modifications can be used 
in order to compute the Fourier coefficients of modular forms which are modular 
under a congruence subgroup of $\SL_2(\ZZ)$. This idea was pioneered by 
Zuckerman \cite{zuckerman1939coefficients}. In this paper we will use such a 
modification in order to compute the Fourier coefficients of a finite product 
of modular forms precomposed with multiplication by different scalar factors 
$\calM\subset \NN$. 
These forms are modular forms under a congruence subgroup of the modular group, 
specifically 
\begin{equation}
\Gamma_0(\lcm (\calM)) = \left\{ \mat{a}{b}{c}{d} \in \SL_2(\ZZ) : c \equiv 0 
\bmod  \lcm (\calM )\right\}.
\end{equation}
So, our result can be obtained using Zuckerman's method applied to 
$\Gamma_0(\lcm(\calM))$. We instead modify 
Rademacher's original method in a slightly different, but ultimately 
equivalent, way so that the 
calculation may be done for many $\eta$-quotients simultaneously. Our 
result is also a special case of recent work by Bringmann and Ono 
\cite{BringmannOno}, but it is easier to simply derive our expressions from 
scratch.

Specifically, we consider $\eta$-quotients, functions $Z(q)$ of the form
\begin{equation}
Z(\tau) = \prod_{m=1}^\infty \eta(m \tau)^{\delta_m}
\end{equation}
where $\{\delta_m\}_{m=1}^\infty$ is a sequence of integers of which only 
finitely many are nonzero. Those with $\delta_m \geq 0$ for all $m$ are often 
called $\eta$-products. These functions have appeared in several different 
contexts. One context, in the vein of their first, is being related to the 
generating 
functions for some partition-like combinatorial quantity. Rademacher's method 
has been applied successfully for a wide variety of these cases. See the work 
of Grosswald \cite{grosswald1958some}\cite{grosswald1960partitions}, 
Haberzetle \cite{haberzetle1941some}, Hagis 
\cite{hagis1962problem}\cite{hagis1963partitions}\cite{hagis1964partitions}\cite{hagis1966some},
 Hua \cite{hua1942number}, Iseki 
\cite{iseki1960some}\cite{iseki1961partitions}, Livingood 
\cite{livingood1945partition}, Niven \cite{niven1940certain}, as well as more 
recent work by Sills \cite{sills2010rademacher}\cite{sills2010rademacher2} and 
others \cite{kiria2011rademacher}\cite{mc2012hardy} for examples. Some of the 
results in these papers are special cases of the main result in this paper, 
although the proof presented in this paper does not work for many of them. It 
breaks down when the modular form of interest has zero weight. The proof 
presented here does work for some of them, for example the result of Sills in 
\cite{sills2010rademacher}.

A second context is as the partition functions for $1/2$-BPS black holes in 
CHL models of string theory, defined originally in \cite{CHLoriginal}, where 
the frame shape of the $\eta$-quotient corresponds to the frame shape of the 
associated $K3$ symplectic automorphism \cite{He:2013lha}\cite{CHLcomposite}. 
The 
constants $d(n)$ then have physical interpretation as the exponential of the 
entropy of a black hole with a given charge.  See 
\cite{0409}\cite{0507}\cite{0502}\cite{SenNotes} for computations of black hole 
entropy in non-reduced rank models using Rademacher series. See 
\cite{CHLcomposite}\cite{JatkarSen}\cite{Govindarajan:2010fu} for derivations 
of CHL $1/2$-BPS black hole partition functions. A classification of all CHL 
models was recently completed \cite{Persson:2015jka}\cite{Paquette:2017gmb} and 
the latter work includes a list of all CHL frame shapes. For Type IIB string 
theory on $K3\times T^2$ or dual models, the Rademacher series coefficients may 
be computed by a gravity path integral. See 
\cite{deBoer:2006vg}\cite{Manschot:2007ha}\cite{Murthy:2009dq}\cite{Dabholkar:2014ema}\cite{Murthy:2015yfa}\cite{Gomes:2017bpi}\cite{Gomes:2017eac}
for some recent work in this regard. $\eta$-products have also been of interest 
in Matheiu moonshine \cite{He:2013lha}\cite{Paquette:2017gmb}. Just as 
Rademacher series were used to 
compute exact black hole entropy from microscopic partition functions for 
non-CHL models, the original motivation for this work was to compute exact 
black hole entropy from microscopic partition functions for CHL models.

In order to present our main formula, we need a few preliminary 
definitions. Let 
\begin{equation}
n_0 = - \frac{1}{24} \sum_{m=1}^\infty m \cdot \delta_m.
\end{equation}
The function $Z(q)\cdot q^{n_0}$ is holomorphic in the open unit disk $\DD$, 
and so we may write 
\begin{equation}
Z(q) = q^{-n_0}\sum_{n=0}^\infty d(n) q^n.
\end{equation}
for some coefficients $d(n)$. From the product formula for $\eta(q)$ in Eq. 
\ref{eq:eta_def}, each $d(n)$ is 
an integer. The main result of this paper is an explicit formula for $d(n)$ for 
a large class of sequences $\{\delta_m\}_{m=1}^\infty$. Let
\begin{equation}
c_1 = -\frac{1}{2}\sum_{m=1}^\infty \delta_m,\quad c_2(k) = \prod_{m=1}^\infty 
\left[ \frac{\gcd(m,k)}{m} \right]^{\delta_m/2}, \quad c_3(k) = - 
\sum_{m=1}^\infty \delta_m\frac{\gcd(m,k)^2}{m},
\label{eq:constant_defs}
\end{equation}
\begin{equation}
A_k(n) = \sum_{\substack{0\leq h < k \\ \gcd(h,k)=1}} \exp\left[-2\pi 
i\left(\frac{h}{k}\cdot n + \frac{1}{2} \sum_{m=1}^\infty \delta_m\cdot s\left( 
\frac{mh}{\gcd(m,k)},\frac{k}{\gcd(m,k)}\right)\ \right) \right].
\label{eq:klooster_defs}
\end{equation}
Finally let $\calM$ be the set of $m$ 
for which $\delta_m$ is 
nonzero. The quantities $c_1,c_2(k),c_3(k)$ are coefficients which appear in 
our 
calculation and formula. The coefficient $c_1$ is the negative of the 
weight of $Z(\tau)$ as a $\Gamma_0(\lcm(\calM))$-modular form. The sums 
$A_k(n)$ 
closely resemble -- and in some 
cases are -- Kloosterman sums.
We will call them Kloosterman-like sums. It can be shown that $A_k(n)$ is real 
for all $k$ and $n$. With these 
definitions in hand, we may state
\begin{theorem} \label{main}
	If $c_1>0$ and the periodic function $g(k):\NN \to \RR$ given by 
	\begin{equation}
	g(k) = \min_{m\in\calM} \left\{ \frac{\gcd(m,k)^2}{m}\right\} - 
	\frac{c_3(k)}{24} \label{eq:g_const}
	\end{equation}
	is non-negative, then for $n\in \{1,2,\ldots\}$ such that $n> n_0$,
	\begin{equation}
	d(n) = 2\pi \left( \frac{1}{24 (n-n_0)}\right)^{\frac{c_1+1}{2}} 
	\sum_{\substack{k=1 \\ c_3(k)>0}}^\infty c_2(k)c_3(k)^{\frac{c_1+1}{2}} 
	A_k(n) 
	k^{-1}I_{1+c_1}\left[ \frac{\pi}{k} \sqrt{\frac{2}{3} c_3(k) (n-n_0)} 
	\right] \label{eq:1_rademacher_main}
	\end{equation}
	where $I_{1+c_1}$ is the $(1+c_1)$th modified Bessel 
	function of the first kind. \hfill $\blacksquare$
\end{theorem}

It is worth spending a moment to comment on the hypotheses of Theorem 
\ref{main}. Rademacher's proof in \cite{RademacherOriginal} works only for 
modular forms of positive dimension, that is negative weight.
Since $c_1$ is the dimension of our modular form under the congruence subgroup 
for 
which it transforms, our constraint $c_1>0$ is analogous to the weight 
constraint of Rademacher. Rademacher and others noted that his formula held for 
other modular forms, including many of weight zero. The various methods of 
proof for these extreme cases seem to be substantially more delicate, relying 
on 
detailed computations involving Kloosterman sums. Rademacher's computation 
of the Fourier coefficients of $J(\tau)$ 
\cite{rademacher1938fourier}\cite{rademacher1939fourier} is a good example.
The story for $\eta$-quotients is analogous. Eq. \ref{eq:1_rademacher_main} 
seems to work for many $\eta$-quotients with $c_1=0$ assuming that the second 
hypothesis regarding $g(k)$ is satisfied. This is not entirely surprising given 
recent work by Duncan and others
\cite{duncan2009rademacher}\cite{cheng2011rademacher}\cite{duncan2015proof}.

Similarly, Rademacher's formula in \cite{RademacherOriginal} includes a sum 
over the polar part of the relevant modular form. For each $k$, each polar term 
gives rise to 
a distinct Bessel function in the Rademacher series. We only get a single 
Bessel function for each $k$, as in Eq. 
\ref{eq:1_rademacher_main}, when the polar part contains one term. 
Analogously, the 
condition that $g(k)\geq 0$ implies that the polar part 
of $Z(\tau)$ at each cusp of $\HH/\Gamma_0(\lcm(\calM))$ contains at most one 
term. If the polar part of $Z(\tau)$ at some cusp of 
$\HH/\Gamma_0(\lcm(\calM))$ contains more than one term, the following 
proof can easily be modified, along with Eq. \ref{eq:1_rademacher_main}, to 
yield the correct expression. It will look like Eq. \ref{eq:1_rademacher_main}, 
but with additional Bessel functions, one for each term in the polar part.

An outline of this short paper is as follows: The proof of Theorem \ref{main} 
is contained in section \ref{sec:proof}. In section \ref{sec:asymp} we 
check that Eq. \ref{eq:1_rademacher_main} has the expected asymptotics. In 
section \ref{sec:numerics} we present some numerics.

\section{Proof of Theorem \ref{main}} \label{sec:proof}

We will use Rademacher's modification of the Hardy-Ramanujan-Littlewood circle 
method to compute $d(n)$. We present some lemmas which are contained in the 
original papers \cite{Rademacher43}\cite{rademacher1938partition} without 
proof. The reader can find proofs of these lemmas in these papers or in many 
expositions, such as \cite{hsu2011partition}, which I personally followed.

As in Rademacher's computation of the Fourier coefficients of $1/\eta(\tau)$, 
we extract 
$d(n)$ by performing a Fourier-Laplace transform:
\begin{equation}
d(n) = \frac{1}{2\pi i} \int_\gamma \frac{Z(q)\cdot q^{n_0}}{q^{n+1}} \dd q 
\label{eq:inv_lap_trans}
\end{equation}
where $\gamma$ is a closed contour winding once around the origin. In essence, 
that is all there is to it. The rest of the proof is simply computing this 
integral.

We will define a sequence of suitable contours $\{\gamma_N\}_{N=1}^\infty$, 
compute the 
integral in Eq. \ref{eq:inv_lap_trans} for $\gamma=\gamma_N$ up to an error 
term, take $N\to \infty$ and show that the error term converges to zero. It is more convenient to define the contours in the $\tau$-plane and then map them into the $q$-plane. We will denote a pullback of some contour $\gamma$ in the $q$-plane to the $\tau$-plane as $\tau(\gamma)$. 

Some preliminary definitions are in other. The $N$th Farey sequence $\calF_N$ 
is the finite sequence containing all irreducible fractions in $[0,1]$ of 
denominator at most $N$ in increasing 
order. The Ford circle $C(h/k)$ associated with the irreducible fraction $h/k$ 
is the circle in the $\tau$-plane with center $h/k+i/2k^2$ and radius 
$1/2k^2$. See Fig. \ref{fig:ford_circles}. We denote by $q(C(h/k))$ the mapping 
of $C(h/k)$ into the $q$-plane. Note that $q(C(0/1)) = q(C(1/1))$. It can be 
shown that the Ford circles corresponding to consecutive 
fractions $h_1/k_2$ and $h_2/k_2$ in some Farey sequence are tangent at the 
point
\begin{equation}
\widetilde{\tau} (h_1/k_1,h_2/k_2) = \frac{h_1k_1+h_2k_2+i}{k_1^2+k_2^2}.
\end{equation}
For irreducible fractions $h_0/h_0<h_1/k_1< h_2/k_2$ with $C(h_0/k_0)$ and 
$C(h_2/k_2)$ tangent to $C(h_1/k_1)$  let 
$\tau(\gamma_{h_0/k_0,h_1/k_1,h_2/k_2})$ be the arc on $C(h_1/k_1)$ from the 
point of tangency with $C(h_0/k_0)$ to the point of tangency with $C(h_2/k_2)$ 
parametrized by arc length. We choose the contour to proceed around the Ford 
circle clockwise so that the arc does not touch the real line. Likewise for 
$h_2/k_2$ such that $C(h_2/k_2)$ is tangent to $C(0/1)$ let 
$\tau(\gamma_{0/1,h_2/k_2})$ be the arc on $C(0/1)$ from the point $+i$ to the 
point of tangency with $C(h_2/k_2)$ parametrized by arc length. Likewise for 
$h_0/k_0$ such that $C(h_0/k_0)$ is tangent to $C(1/1)$ let 
$\tau(\gamma_{h_0/k_0,1/1})$ be the arc on $C(0/1)$ from the point of tangency 
with $C(h_0/k_0)$ to the point $1+i$ parametrized by arc length. If we specify 
an $N\in \{1,2,\ldots\}$ then we may define $\tau(\gamma_{N,h/k})$ for $h/k \in 
\calF_N$ to be 
\begin{equation}
\tau(\gamma_{N,h/k}) = 
\begin{cases}
\tau(\gamma_{0/1,h_2/k_2}) & (h/k=0/1) \\
\tau(\gamma_{h_0/k_0,h/k,h_2/k_2}) & (k\neq 1) \\
\tau(\gamma_{h_0/k_0,1/1}) & (h/k = 1/1)
\end{cases}
\end{equation}
where $h_0/k_0$ is the element in $\calF_N$ immediately before $h/k$ if such an 
element exists and $h_2/k_2$ is the element in $\calF_N$ immediately after 
$h/k$ if such an element exists.

\begin{figure}[t!]
	\centering
	\includegraphics[width = .9\textwidth]{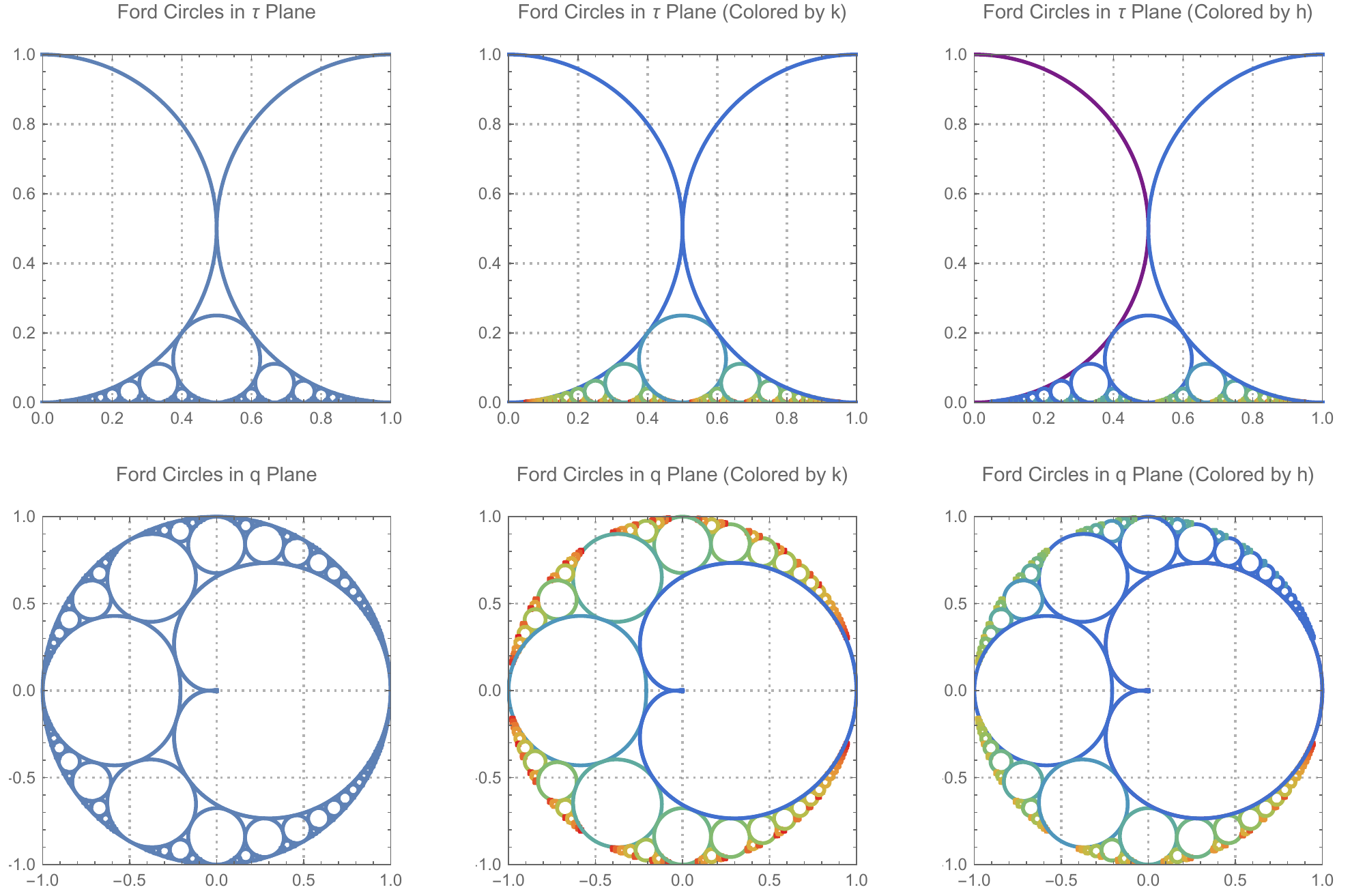}
	\caption{The first few Ford circles in the $\tau$ and $q$ planes, with 
	various color schemes.}
	\label{fig:ford_circles}
\end{figure} 

We define $\tau(\gamma_N)$ to be the concatenation in order of each 
$\tau(\gamma_{N,h/k})$ for $h/k \in \calF_N$. The contour $\gamma_N$ is then 
the mapping of $\tau(\gamma_N)$ into the $q$-plane. This is a concatenation of 
the contours $\gamma_{N,h/k}$ for $h/k \in \calF_N$. We redefine 
$\gamma_{N,0/1}$ to be the concatenation of what we used to call 
$\gamma_{N,0/1}$ and $\gamma_{N,1/1}$. These contours meet in the $q$-plane. 
The full contour $\gamma_N$ is piecewise smooth and has winding number one 
about the origin. See Fig. \ref{fig:rademacher_contour} for visualizations of 
the Rademacher contour $\gamma_N$ and $\tau(\gamma_N)$ for various values of 
$N$.

We now split up the contour integral in Eq. \ref{eq:inv_lap_trans} into a sum of integrals over subcontours:
\begin{equation}
d(n) = \frac{1}{2\pi i} \sum_{k=1}^N \sum_{\substack{1\leq h < k \\ (h,k)=1}} 
\int_{\gamma_{N,h/k}} \frac{Z(q)}{q^{(n-n_0)+1}} \dd q 
\label{eq:subcontour_integral}
\end{equation}
where $(h,k)$ is shorthand for $\gcd(h,k)$. It is convenient to change 
coordinates within each subcontour integral in order to write them as integrals 
over similar contours. Note that for irreducible $h/k$ the coordinate 
transformation 
\begin{equation}
z = - i k^2 \left( \tau - \frac{h}{k} \right) \text{ or equivalently } \tau = i 
\cdot \frac{z}{k^2} + \frac{h}{k}
\label{eq:z_trans_def}
\end{equation}
maps the Ford circle $C(h/k)$ in the $\tau$-plane to the circle $B_{1/2}(1/2)$ 
in the $z$-plane with center $1/2$ and radius $1/2$. See Fig. 
\ref{fig:z_plane_circles}. The point $\widetilde{\tau} (h/k,h_2/k_2)$ is mapped 
to the point 
\begin{equation}
\widetilde{z}_{h/k}(h_2/k_2) = \frac{k^2}{k^2+k_2^2}  + i \left( hk - 
\frac{k^2}{k^2+k_2^2} (hk+h_2k_2)\right).
\end{equation}
We moved the $h/k$ into the subscript of $\widetilde{z}$ to emphasize that the 
coordinate transformation depends on $h/k$ and that, for this reason, unlike 
$\widetilde{\tau}$, $\widetilde{z}$ is not symmetric under interchanging its 
arguments.
The contour $\tau(\gamma_{N,h/k})=\tau(\gamma_{h_0/k_0,h/k,h_2/k_2})$ is mapped 
to an arc along $B_{1/2}(1/2)$ from $\widetilde{z}_{h/k}(h_0/k_0)$ to 
$\widetilde{z}_{h/k}(h_2/k_2)$, specifically the arc which does not contain the 
origin. Likewise, the contours $\tau(\gamma_{0/1,h_2/k_2})$ and 
$\tau(\gamma_{h_0/k_0,1/1})$ are mapped together to an arc along $B_{1/2}(1/2)$ 
from $\widetilde{z}_{1/1}(h_0/k_0)$ to $\widetilde{z}_{0/1}(h_2/k_2)$, also the 
arc which does not contain the origin. Let $\widetilde{z}_{1,N,h/k}$ be 
$\widetilde{z}_{h/k}(h_0/k_0)$ where $h_0/k_0$ is the element of $\calF_N$ 
immediately before $h/k$ if $k\neq 1$ and $\widetilde{z}_{1/1}(h_0/k_0)$ where 
$h_0/k_0$ is the element of $\calF_N$ immediately before $1/1$ if $k=1$. For 
irreducible $h/k$ except $1/1$ let $\widetilde{z}_{2,N,h/k}$ be 
$\widetilde{z}_{h/k}(h_2/k_2)$ where $h_2/k_2$ is the element of $\calF_N$ 
immediately after $h/k$. It can be checked that 
\begin{equation}
\widetilde{z}_{1,N,h/k} = \frac{k}{k^2+k_0^2}(k+ik_0)
\end{equation}
\begin{equation}
\widetilde{z}_{2,N,h/k} = \frac{k}{k^2+k_2^2}(k-ik_2)
\end{equation}
where $h_0/k_0<h/k<h_2/k_2$ are consecutive fractions in $\calF_N$ or if 
$h/k=0/1$ and $h_0/k_0$ is immediately preceding $1/1$ or if $h/k=1/1$ and 
$h_2/k_2$ is immediately following $0/1$. See Fig. \ref{fig:z_plane_circles}.

\begin{figure}[t!]
	\centering
	\includegraphics[width=.95\textwidth]{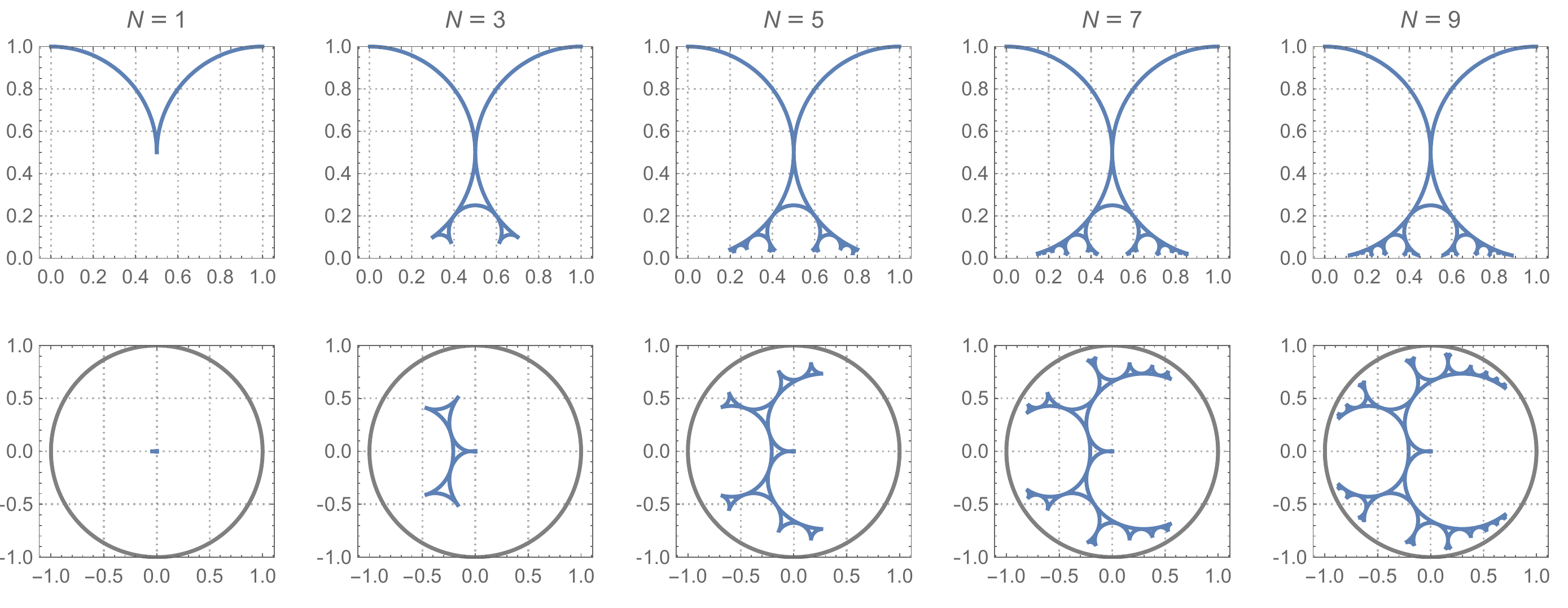}
	\caption{The Rademacher contour $\gamma_N$ for several different values of 
		$N$, in the $\tau$ and $q$ planes.}	
	\label{fig:rademacher_contour}
\end{figure} 

Performing the coordinate 
transformations to Eq. \ref{eq:subcontour_integral}
\begin{equation}
d(n) = i \sum_{k=1}^N k^{-2} \sum_{\substack{0\leq h < k \\ (h,k)=1}} 
\int_{z(\gamma_{N,h/k})} Z \left( \exp\left( 2\pi \left[ i\cdot \frac{h}{k} - 
\frac{z}{k^2} \right] \right) \right) \exp\left( 2 \pi (n-n_0)\left( 
\frac{z}{k^2} -i\cdot \frac{h}{k} \right) \right) \dd z. 
\end{equation}

Here $z(\gamma_{N,h/k})$ is the mapping of $\tau(\gamma_{N,h/k})$ into the 
$z$-plane. That is, $z(\gamma_{N,h/k})$ is the arc of $B_{1/2}(1/2)$ which 
avoids the origin and is from $\widetilde{z}_{1,N,h/k}$ to 
$\widetilde{z}_{2,N,h/k}$. Tracing through the definitions of $z$ and $Z(q)$, 
the 
integrands above are holomorphic in the right-half of the $z$-plane. We may 
therefore deform our subcontours from arcs on $B_{1/2}(1/2)$ to chords through 
$B_{1/2}(1/2)$. These chords begin at $\widetilde{z}_{1,N,h/k}$ and end at 
$\widetilde{z}_{2,N,h/k}$. We denote these chords as $ 
\overline{z_1z_2}(N,h/k)$. See Fig. \ref{fig:z_plane_circles}. So,
\begin{equation}
d(n) = i \sum_{k=1}^N k^{-2} \sum_{\substack{0\leq h < k \\ (h,k)=1}} 
\int_{\overline{z_1z_2}(N,h/k)} Z \left( \exp\left( 2\pi \left[ i\cdot 
\frac{h}{k} - 
\frac{z}{k^2} \right] \right) \right) \exp\left( 2 \pi (n-n_0)\left( 
\frac{z}{k^2} -i\cdot \frac{h}{k} \right) \right) \dd z. 
\label{eq:subcontour_integral_straightened}
\end{equation}

Before we proceed, we state two geometric lemmas. The 
first 
concerns the 
properties of the chords  $\overline{z_1z_2}(N,h/k)$ and the second concerns 
the properties of arcs on $B_{1/2}(1/2)$. Proofs of these are contained in 
\cite{hsu2011partition}.

\begin{lemma}
	The chord $\overline{z_1z_2}(N,h/k)$ has length at most $2\sqrt{2} k/N$ 
	and on this chord $|z| \leq \sqrt{2} k/N$. $\hfill\blacksquare$ 
	\label{lemma:chordprops}
\end{lemma}

\begin{lemma}
	In $B_{1/2}(1/2)/\{0\}$, $\Re(z)\leq 1$ and $\Re(1/z)\geq 1$ with 
	$\Re(1/z)=1$ on the circle itself. On the arcs from $0$ to 
	$\widetilde{z}_{1,N,h/k}$ and $\widetilde{z}_{2,N,h/k}$ to $0$,  $|z| \leq 
	\sqrt{2} k/N$
	and the length of these arcs is at most $\pi \sqrt{2} k/N$. 
	$\hfill\blacksquare$
	\label{lemma:arcprops}
\end{lemma}

By the previous two lemmas, for fixed $h/k$ as $N\to\infty$ the chords 
$\overline{z_1z_2}(N,h/k)$ get shorter and nearer to the origin. As $z$ 
approaches the origin, $\tau = (h/k)+i(z/k^2)$ approaches $h/k$. We can 
calculate the asymptotics of $\eta(\tau)$ as $\tau$ approaches $+ i \infty$ 
from the definition of $\eta(\tau)$. We can then calculate the asymptotics of 
$Z(\tau)$ near $h/k$ in terms of the asymptotics of each  $\eta(m\tau)$ near 
$+i\infty$ using the modularity properties of $\eta(\tau)$. These asymptotics 
are sufficiently simple to integrate them. This is the key 
insight in the Hardy-Littlewood-Ramanujan circle method. We now turn to 
expressing $\eta( m \tau)$ for $m\in\{1,2,\ldots\}$ near $h/k$ in terms of 
$\eta(\tau)$ near $+i\infty$.

\begin{figure}[t!]
	\centering
	\includegraphics[width=.8\textwidth]{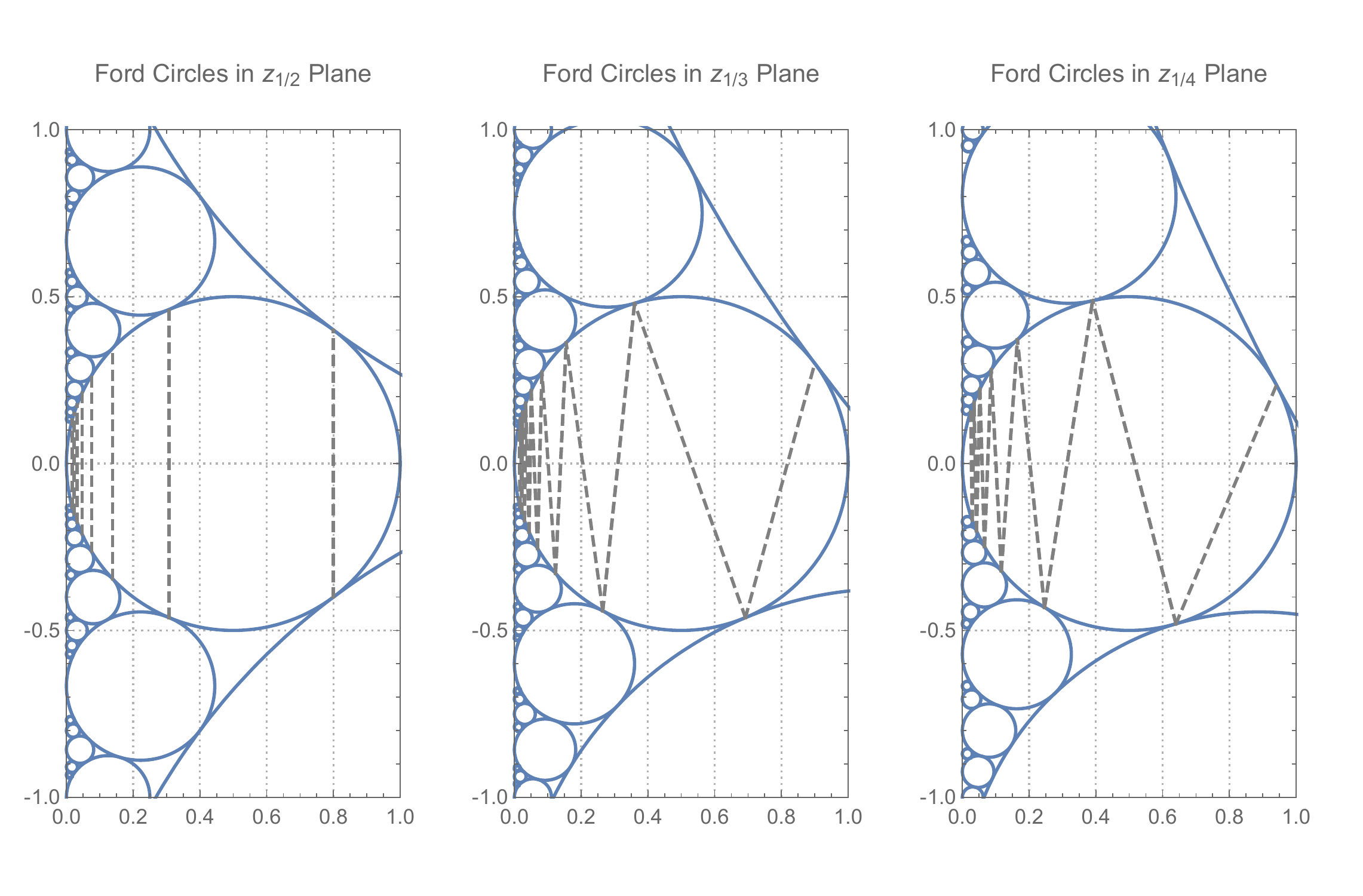}
	\caption{The first few Ford circles in the $z$-plane using the 
		transformation in Eq. 
		\ref{eq:z_trans_def} for $h/k=1/2$ (left), $h/k=1/3$ (center), and 
		$h/k=1/4$ (right). The chords $\overline{z_1z_2}(N,h/k)$ for various 
		$N$ are denoted by 
		gray dashed lines.}
	\label{fig:z_plane_circles}
\end{figure} 

For irreducible fraction $h/k \in[0,1]$, we may find by the Euclidean 
algorithm some integer $H(m,h,k)=H$ such that 
\begin{equation}
m h H \equiv - \gcd(m,k)\mod k.
\end{equation}
It follows that the matrix $M=\begin{pmatrix}a&b\\c&d\end{pmatrix}$ is in 
$\SL_2(\ZZ)$, where
\begin{equation}
	a = H, \quad b = - \frac{1}{k}(mhH+\gcd(m,k)), \quad c = 
	\frac{k}{\gcd(m,k)}, \quad d= -\frac{mh}{\gcd(m,k)}.
\end{equation}
As a member of the modular group, this matrix induces an action on the 
upper-half of the complex plane, given by $M(\tau)= \frac{a\tau+b}{c\tau+d}$ 
for any $\tau\in\HH$. 
Under this action,
\begin{equation}
	M(m\tau) = M\left( m\left[\frac{h}{k}+i \cdot \frac{z}{k^2} \right] 
	\right)=  
	\frac{\gcd(m,k)}{k}\left( H + i \cdot \frac{k}{mz}\gcd(m,k) \right) = 
	\frac{a\tau+b}{c\tau+d}.
\end{equation}
As $z$ approaches $0$ -- and therefore $\tau$ approaches $h/k$ -- the right 
hand side goes 
to 
positive imaginary infinity, as desired. Using the modular transformation 
properties of 
$\eta(\tau)$, 
\begin{equation}
	\eta(m\tau) = [\epsilon(a,b,c,d) (cm\tau+d)^{1/2}]^{-1}\;\eta\left( 
	\frac{am\tau+b}{cm\tau+d} \right).
\end{equation}
In this case, $cm\tau+d = imz/k\gcd(m,k)$ and by Eq. \ref{eq:mult_system}
\begin{equation}
	\epsilon(a,b,c,d)= \exp\left(\pi i \left( 
	\frac{H}{12k}\gcd(m,k)-\frac{mh}{12 
	k}+s\left(\frac{mh}{\gcd(m,k)},\frac{k}{\gcd(m,k)} \right)-\frac{1}{4} 
	\right) \right),
\end{equation} 
so that
\begin{multline}
	\eta\left( m\left(\frac{h}{k}+i \cdot\frac{z}{k^2} \right)\right) = 
	\exp\left(\pi i \left( -\frac{H}{12k}\gcd(m,k)+\frac{mh}{12 
	k}-s\left(\frac{mh}{\gcd(m,k)},\frac{k}{\gcd(m,k)} \right) \right) \right)  
	\\ \times \sqrt{\frac{k\gcd(m,k)}{mz}} 
	\:\eta\left(\frac{\gcd(m,k)}{k}\left( H + i\cdot \frac{k}{mz}\gcd(m,k) 
	\right) 
	\right).
	\label{eq:full_eta_trans}
\end{multline}
The constant $H=H(m,h,k)$ depends implicitly on $m$, $h$, and $k$. Combining 
Eq. \ref{eq:full_eta_trans} for all values of $m$
\begin{equation}
	Z\left( \frac{h}{k}+i\cdot \frac{z}{k^2}\right) = \xi(h,k) \omega(h,k) 
	\cdot 
	z^{c_1}\prod_{m=1}^\infty \eta\left(\frac{\gcd(m,k)}{k}\left( H + i \cdot
	\frac{k}{mz}\gcd(m,k) \right) \right)^{\delta_m} ,
\end{equation}
where $c_1$ was defined in Eq. \ref{eq:constant_defs} and we 
defined
\begin{equation}
	\xi(h,k) = \prod_{m=1}^\infty \left[ 
	\sqrt{\frac{k\gcd(m,k)}{m}}\exp\left(\pi i 
	\left( \frac{mh}{12 k}-s\left(\frac{mh}{\gcd(m,k)},\frac{k}{\gcd(m,k)} 
	\right) \right)\right) \right]^{\delta_m}
	\label{eq:xi_def}
\end{equation}
and
\begin{equation}
	\omega(h,k) = \prod_{m=1}^\infty \exp\left(-\frac{\pi i}{12 k} H 
	\delta_m\gcd(m,k)\right).
	\label{eq:omega_def}
\end{equation}
Plugging in the previous formulas into 
Eq. \ref{eq:subcontour_integral_straightened}:
\begin{multline}
	d(n) = i\sum_{k=1}^N  k^{-2} \sum_{\substack{0\leq h < k \\ (h,k)=1}}
	\xi(h,k)\omega(h,k) \int_{\overline{z_1z_2}(N,h/k)} z^{c_1} 
	\left[ \prod_{m=1}^\infty \eta\left(\frac{\gcd(m,k)}{k}\left( H + i 
	\frac{k}{mz}\gcd(m,k) \right) \right)^{\delta_m} \right] \\ \times
	\exp\left( 2\pi 
	(n-n_0)\left( \frac{z}{k^2}-i \cdot\frac{h}{k}\right)\right) \dd z.
\end{multline}
We expect to be able to replace each $\eta(q)$ in the integrand with 
$q^{1/24}$ and accrue a total $o(1)$ error as $N\to\infty$.  Defining
\begin{equation}
\operatorname{Er}(N) = i\sum_{k=1}^N \sum_{\substack{0\leq h < k \\ 
(h,k)=1}} k^{-2} \xi(h,k) \omega(h,k) 
\int_{\overline{z_1z_2}(N,h/k)}z^{c_1}\exp\left(  2\pi  (n-n_0) \left(  
\frac{z}{k^2}-i\frac{h}{k} \right) \right) \Delta_{h,k}(z)\dd{z},
\label{eq:error_exp}
\end{equation}
where $\Delta_{h,k}(z)$ is the difference between the $\eta$-quotient and its 
leading order asymptotics
\begin{multline}
\Delta_{h,k}(z) = \left[\prod_{m=1}^\infty \eta\left( 
\frac{\gcd(m,k)}{k}\left(H + i 
\frac{k}{mz}\gcd(m,k) \right) \right)^{\delta_m} \right]-\\ 
\left[\prod_{m=1}^\infty 
\exp\left( \frac{\pi i}{12}\frac{\gcd(m,k)}{k}\left(H + i \frac{k}{mz}\gcd(m,k) 
\right) \right)^{\delta_m} \right],
\label{eq:delta_def}
\end{multline}
we can write
\begin{equation}
d(n) = \operatorname{Er}(N)+i\sum_{k=1}^N k^{-2} \sum_{\substack{0\leq h < k \\ 
		(h,k)=1}} \xi(h,k)e^{-2\pi i (n-n_0) 
\cdot \frac{h}{k}}\int_{\overline{z_1z_2}(N,h/k)} z^{c_1}\exp\left[ \frac{2\pi 
(n-n_0)}{k^2}z 
+ 
\frac{\pi}{12}\frac{c_3(k)}{z}\right]\dd{z}.
\label{eq:dn_with_er}
\end{equation}
The second term is the result of replacing each $\eta$-function by the 
appropriate 
asymptotics, and the first term is the accumulated error from doing so.
We wish to show that $\lim_{N\to\infty}\text{Er}(N) = 0$.

Our first task is to bound $\Delta_{h,k}(z)$. We will show that 
$\Delta_{h,k}(z)=O(1)$ in $B_{1/2}(1/2)$ where the bound does not depend on 
$h$, $k$, or $z$. The 
fact that $\Delta_{h,k}(z)$ does not blow up at the origin expresses the fact 
that the polar component of $Z(\tau)$ near $h/k$ contains at most one term, 
that which we are approximating it by. Let $\tilde{\eta}(\tau)= 
\eta(\tau)/q^{1/24}$, so that
\begin{multline}
\Delta_{h,k}(z)= \left[\prod_{m=1}^\infty \exp\left( \frac{\pi 
	i}{12}\frac{\gcd(m,k)}{k}\left(H + i \frac{k}{mz}\gcd(m,k) \right) 
\right)^{\delta_m} \right]\\ \times \left[ \left[\prod_{m=1}^\infty 
\tilde{\eta}\left( 
\frac{\gcd(m,k)}{k}\left(H + i \frac{k}{mz}\gcd(m,k) \right) 
\right)^{\delta_m} \right]-1\right].
\end{multline}
Taking the norm,
\begin{equation}
|\Delta_{h,k}(z)|= \exp\left[ \frac{\pi}{12 } c_3(k) \Re\left( 
\frac{1}{z}\right) \right] \times \left| \prod_{m=1}^\infty 
\tilde{\eta}\left( 
\frac{\gcd(m,k)}{k}\left(H + i \frac{k}{mz}\gcd(m,k) \right) 
\right)^{\delta_m} -1\right|.\label{eq:before_lem_eta}
\end{equation}
If $c_3(k)>0$, the first term on the right hand side can be large for some $z$ 
along $\overline{z_1z_2}(N,h/k)$. However, we expect the second term on the 
right hand side to be small. This is the content of the following lemma.
\begin{lemma}
	For some constant $D$, which may depend only on $\{\delta_m\}_{m=1}^\infty$,
	for $z$ in $\overline{B_{1/2}(1/2)}/\{0\}$,
	\begin{equation}
	\left| \prod_{m=1}^\infty \tilde{\eta}\left( \frac{\gcd(m,k)}{k}\left(H + i 
	\frac{k}{mz}\gcd(m,k) \right) \right)^{\delta_m} -1\right| \leq D 
	\exp\left( -2\pi \Re\left(\frac{1}{z}\right) \min_{m\in\mathcal{M}} \left\{ 
	\frac{\gcd(m,k)^2}{m}\right\} \right). 
	\end{equation}
	\hfill $\blacksquare$
	\label{lem:etabound}
\end{lemma}
\begin{proof}[Proof of Lemma \ref{lem:etabound}]
	First consider
	\begin{equation}
	\tilde{\eta}\left( \frac{\gcd(m,k)}{k}\left(H + i \frac{k}{mz}\gcd(m,k) 
	\right) \right)^{-1},
	\end{equation}
	which by definition is 
	\begin{equation}
	\sum_{j=0}^\infty p(j) \exp\left(-2\pi j \cdot 
	\frac{\gcd(m,k)^2}{m}\Re\left(\frac{1}{z}\right) \right) \omega(h,m,k,z)^j
	\label{eq:lem_geoseries}
	\end{equation}
	where $p(j)$ is Euler's partition function and $\omega(h,m,k,z)=\omega$ is 
	a phase factor. Recall the following upper bound for $p(j)$, which can be 
	derived by purely 
	classical methods \cite{Apostol}
	\begin{equation}
	p(j) = O\left( \exp\left(\pi\sqrt{\frac{2}{3} \cdot j}\right) \right).
	\label{eq:clas_part_bond}
	\end{equation} 
	The tail of the series in Eq. \ref{eq:lem_geoseries} is therefore bounded 
	above by a convergent geometric series. To be more precise, consider the 
	sum of all but the first two terms in Eq. \ref{eq:lem_geoseries}
	\begin{equation}
	\exp\left(-2\pi  \frac{\gcd(m,k)^2}{m}\Re\left(\frac{1}{z}\right) 
	\right)\omega \sum_{j=2}^\infty p(j) \exp\left(-2\pi (j-1) 
	\frac{\gcd(m,k)^2}{m}\Re\left(\frac{1}{z}\right) \right) \omega^{j-1}.
	\label{eq:sumofallbutfirsttwoterms}
	\end{equation}
	Using the triangle inequality and Lemma \ref{lemma:arcprops}, 
	\begin{equation}
	\left| \sum_{j=2}^\infty p(j) \exp\left(-2\pi (j-1) 
	\frac{\gcd(m,k)^2}{m}\Re\left(\frac{1}{z}\right) \right) \omega^{j-1} 
	\right| 
	\leq \sum_{j=1}^\infty p(j) \exp\left(-2\pi (j-1)
	\frac{\gcd(m,k)^2}{m}\right). 
	\end{equation}
	Using the classical bound on $p(j)$ in Eq. \ref{eq:clas_part_bond},
	\begin{equation}
	\sum_{j=1}^\infty p(j) \exp\left(-2\pi (j-1)
	\frac{\gcd(m,k)^2}{m}\right) 
	\leq C \sum_{j=2}^\infty  \exp\left(\pi\sqrt{\frac{2}{3} \cdot j}-2\pi 
	(j-1) \cdot
	\frac{\gcd(m,k)^2}{m} \right)
	\end{equation}
	for some absolute constant $C$.
	Therefore, since $\gcd(m,k) / m \geq 1/m$,
	\begin{equation}
	\left| \sum_{j=2}^\infty p(j)\exp\left(-2\pi (j-1) 
	\frac{\gcd(m,k)^2}{m}\Re\left(\frac{1}{z}\right) \right) \omega^{j-1} 
	\right| 
	\leq C \sum_{j=2}^\infty  \exp\left(\pi\sqrt{\frac{2}{3} \cdot 
		j}-\frac{2\pi}{m} 
	(j-1)  
	\right). 
	\end{equation}
	This is bounded above by a convergent geometric series, so that for some 
	constant $C_m>0$ which is dependent only on $m$,
	\begin{equation}
	\left| \sum_{j=2}^\infty p(j) \exp\left(-2\pi (j-1) 
	\frac{\gcd(m,k)^2}{m}\Re\left(\frac{1}{z}\right) \right) \omega^{j-1} 
	\right| 
	\leq C_m.
	\end{equation}
	It follows that Eq. \ref{eq:sumofallbutfirsttwoterms} is bounded 
	above in magnitude by
	\begin{equation}
	C_m \exp\left(-2\pi  \frac{\gcd(m,k)^2}{m}\Re\left(\frac{1}{z}\right) 
	\right).
	\end{equation}
	Therefore, replacing the tail in Eq. \ref{eq:lem_geoseries} with this bound,
	\begin{equation}
	\tilde{\eta}\left( \frac{\gcd(m,k)}{k}\left(H + i \frac{k}{mz}\gcd(m,k) 
	\right) \right)^{-1}= 1 + O(1)\exp\left(-2\pi  
	\frac{\gcd(m,k)^2}{m}\Re\left(\frac{1}{z}\right) \right),
	\label{eq:lem_recip_bound}
	\end{equation}
	where the $O(1)$ term satisfies $|O(1)| \leq C_m+1$. Taking the reciprocal 
	of Eq. \ref{eq:lem_recip_bound}, and using that $\Re(1/z)\geq 1$ within 
	$B_{1/2}(1/2)$, yields
	\begin{equation}
	\tilde{\eta}\left( \frac{\gcd(m,k)}{k}\left(H + i \frac{k}{mz}\gcd(m,k) 
	\right) \right) = 1 +O(1) \exp\left(-2\pi  
	\frac{\gcd(m,k)^2}{m}\Re\left(\frac{1}{z}\right) \right)
	\label{eq:lem_bound}
	\end{equation}
	for some other $O(1)$ term which can be bounded in magnitude depending only 
	on $m$.
	Taking the appropriate product of Eq. \ref{eq:lem_recip_bound} and Eq. 
	\ref{eq:lem_bound} for all $m$ yields 
	\begin{equation}
	\prod_{m=1}^\infty \tilde{\eta}\left( \frac{\gcd(m,k)}{k}\left(H + i 
	\frac{k}{mz}\gcd(m,k) \right) \right)^{\delta_m} = 1 +O(1)\exp\left( - 2 
	\pi \min_{m\in\mathcal{M}}\left\{\frac{\gcd(m,k)^2}{m}\right\} 
	\Re\left(\frac{1}{z}\right) \right),
	\end{equation}
	where $\mathcal{M}$ is the (finite) set of all $m\in\NN$ such that 
	$\delta_m$ is nonzero, and the $O(1)$ term is bounded in magnitude 
	depending 
	only on 
	$\{\delta_m\}_{m=1}^\infty$.
	Consequently, for some constant $D>0$ depending only on 
	$\{\delta_m\}_{m=1}^\infty$
	\begin{equation}
	\left|\prod_m\tilde{\eta}\left( \frac{\gcd(m,k)}{k}\left(H + i 
	\frac{k}{mz}\gcd(m,k) \right) \right)^{\delta_m} - 1 \right| \leq D 
	\exp\left( - 2 \pi 
	\min_{m\in\mathcal{M}}\left\{\frac{\gcd(m,k)^2}{m}\right\} 
	\Re\left(\frac{1}{z}\right) \right).
	\end{equation}
\end{proof}
Using Lemma \ref{lem:etabound},
\begin{equation}
|\Delta_{h,k}(z)| \leq D \exp\left[-2\pi \Re\left(\frac{1}{z}\right)\left(  
\min_{m\in\mathcal{M}} \left\{ 
\frac{\gcd(m,k)^2}{m}\right\} - \frac{c_3(k)}{24}\right) \right]= D 
\exp\left[-2\pi 
\Re\left(\frac{1}{z}\right) g(k) \right].
\label{eq:delta_bd_wg}
\end{equation}
One of the hypotheses of Theorem \ref{main} is that the function $g(k)$ is 
non-negative. 
Then, using Lemma \ref{lemma:arcprops}, we can bound 
\begin{equation}
|\Delta_{h,k}(z)| \leq D \exp \left[- 2\pi \min\{g(k) : k=1,\ldots,\lcm 
(\calM)\} \right],
\label{eq:delta_bound}
\end{equation}
and the constant on the right hand side depends only on 
$\{\delta_m\}_{m=1}^\infty$. We redefine $D$ to be this constant.
Using Eq. \ref{eq:delta_bound} and Lemma 
\ref{lemma:chordprops}, the integral in Eq. \ref{eq:error_exp} is 
bounded above in magnitude:
\begin{multline}
\left|\int_{\overline{z_1z_2}(N,h/k)}z^{c_1}\exp\left(  2\pi  (n-n_0) \left(  
\frac{z}{k^2}-i\frac{h}{k} \right) \right) \Delta_{h,k}(z)\dd{z}\right|\\ \leq 
2D\left( \sqrt{2}\cdot 
\frac{k}{N} \right)^{c_1+1} \exp\left( 
2\sqrt{2} 
\pi (n-n_0) \cdot \frac{k}{N} \right).
\end{multline}
Substituting this into the definition of $\text{Er}(N)$ in Eq. 
\ref{eq:error_exp}, for some 
constant $C$ depending only on 
$\{\delta_m\}_{m=1}^\infty$,
\[
|\text{Er}(N)| \leq C e^{2\sqrt{2}\pi n} N^{-(c_1+1)}\sum_{k=1}^N 
\sum_{\substack{0\leq h < k \\ (h,k) = 1}} k^{-1} .
\]
Since there are at most $k$ terms in the inner 
sum and $N$ terms in the outer sum, we can bound this 
\begin{equation}
	\label{eq:erboundabc}
	|\text{Er}(N)| \leq C e^{2\sqrt{2}\pi n} N^{-c_1}. 
\end{equation}
Since $c_1>0$, this shows that $\lim_{N\to\infty} \text{Er}(N) = 0$, as 
desired. Referring back to Eq. \ref{eq:dn_with_er}, we have shown that
\begin{equation}
d(n) = o(1)+i\sum_{k=1}^N k^{-2}\sum_{\substack{0\leq h < k \\ (h,k)=1}}  
\xi(h,k)e^{-2\pi i (n-n_0)
h/k}\int_{\overline{z_1z_2}(N,h/k)} 
z^{c_1}\exp\left[ \frac{2\pi (n-n_0)}{k^2}z + 
\frac{\pi}{12}\frac{c_3(k)}{z}\right]\dd{z}.
\end{equation}
We now deform our contours back to arcs along $B_{1/2}(1/2)$:
\begin{equation}
d(n) = o(1)+i\sum_{k=1}^Nk^{-2} \sum_{\substack{0\leq h < k \\ (h,k)=1}}
\xi(h,k)e^{-2\pi i 
(n-n_0) h/k}\int_{z(\gamma_{N,h/k})} 
z^{c_1}\exp\left[ \frac{2\pi (n-n_0)}{k^2}z + 
\frac{\pi}{12}\frac{c_3(k)}{z}\right]\dd{z}.
\end{equation}
Our next goal is to show that the main term on the right hand side above,
\begin{align}
i\sum_{k=1}^Nk^{-2}\sum_{\substack{0\leq h < k \\ (h,k)=1}} \xi(h,k)e^{-2\pi i 
(n-n_0) 
h/k}\int_{z(\gamma_{N,h/k})} 
z^{c_1}\exp\left[ \frac{2\pi (n-n_0)}{k^2}z + 
\frac{\pi}{12}\frac{c_3(k)}{z}\right]\dd{z},
\end{align}
differs from
\begin{equation}
i\sum_{k=1}^Nk^{-2}\sum_{\substack{0\leq h < k \\ (h,k)=1}} \xi(h,k)e^{-2\pi i 
(n-n_0) 
h/k}\int_{B_{1/2}(1/2)} 
z^{c_1}\exp\left[ \frac{2\pi (n-n_0)}{k^2}z + 
\frac{\pi}{12}\frac{c_3(k)}{z}\right]\dd{z} 
\end{equation}
by an $o(1)$ term as $N\to\infty$. The contour in the second expression is 
traversed clockwise. In 
other words, we may replace our integrals over incomplete arcs of 
$B_{1/2}(1/2)$ by integrals over the complete circle $B_{1/2}(1/2)$ and only 
accrue a total $o(1)$ 
error as $N\to\infty$.
The former is the latter minus $J_1+J_2$, where $J_1=J_1(N)$ and $J_2=J_2(N)$ 
are defined by
\begin{equation}
J_1 = i\sum_{k=1}^N\sum_{\substack{0\leq h < k \\ (h,k)=1}}k^{-2} 
\xi(h,k)e^{-2\pi i (n-n_0) 
h/k}\int_{0}^{\tilde{z}_1(N,h/k)} 
z^{c_1}\exp\left[ \frac{2\pi (n-n_0)}{k^2}z + 
\frac{\pi}{12}\frac{c_3(k)}{z}\right]\dd{z},
\end{equation}
\begin{equation}
J_2 = i\sum_{k=1}^N\sum_{\substack{0\leq h < k \\ (h,k)=1}}k^{-2} 
\xi(h,k)e^{-2\pi i (n-n_0) 
h/k}\int_{\tilde{z}_2(N,h/k)}^0 
z^{c_1}\exp\left[ \frac{2\pi (n-n_0)}{k^2}z + 
\frac{\pi}{12}\frac{c_3(k)}{z}\right]\dd{z} .
\end{equation}
The integrals are on arcs of $B_{1/2}(1/2)$. 
We bound $J_1$. Using the bound $|\xi(h,k)| \leq k^{-c_1}$,
\begin{equation}
|J_1| \leq \sum_{k=1}^N\sum_{\substack{0\leq h < k \\ 
(h,k)=1}}k^{-(2+c_1)}\left|\int_{0}^{\tilde{z}_1(N,h/k)} z^{c_1}\exp\left[ 
\frac{2\pi (n-n_0)}{k^2}z + \frac{\pi}{12}\frac{c_3(k)}{z}\right]\dd{z}\right|.
\end{equation}
Using Lemma \ref{lemma:arcprops},
\begin{equation}
|J_1| \leq \pi 
\sum_{k=1}^N\sum_{\substack{0\leq h < k \\ (h,k)=1}}k^{-(2+c_1)} 
 \left( 
\frac{\sqrt{2} k}{N} \right)^{c_1+1} \exp\left[ \frac{2\pi (n-n_0)}{k^2} + 
\frac{\pi}{12} c_3(k) \right].
\end{equation}
So, for some constant $C$ depending only on $\{\delta_m\}_{m=1}^\infty$,
\begin{equation}
|J_1| \leq C\sum_{k=1}^N\sum_{\substack{0\leq h < k \\ (h,k)=1}} k^{-1} 
\left(\frac{1}{N}\right)^{c_1+1}   \exp\left[ 
2\pi (n-n_0)  \right].
\end{equation}
Since the outer sum is over at most $N$ terms and the inner sum is over at most 
$k$ terms,
\begin{equation}
|J_1| \leq CN^{-c_1}  \exp\left[ 2\pi (n-n_0) \right].
\end{equation}
Since $c_1>0$, $J_1 = o(1)$, as desired. An identical 
argument yields $J_2 = o(1)$.
Combining all of the previous results,
\begin{equation}
d(n) = o(1)+i\sum_{k=1}^N k^{-2}  \sum_{\substack{0\leq h < k \\ \gcd(h,k)=1}}
\xi(h,k) e^{-2\pi i (n-n_0) \cdot \frac{h}{k}}\int_{B_{1/2}(1/2)} z^{c_1} 
\exp\left[ \frac{2\pi (n-n_0)}{k^2}z + \frac{\pi}{12} \frac{c_3(k)}{z} 
\right]\dd{z}.
\end{equation}
Taking $N\to\infty$,
\begin{equation}
d(n) = i\sum_{k=1}^\infty k^{-2} \sum_{\substack{0\leq h < k \\ \gcd(h,k)=1}}  
\xi(h,k) e^{-2\pi i (n-n_0) \cdot \frac{h}{k}}\int_{B_{1/2}(1/2)}  z^{c_1} 
\exp\left[ \frac{2\pi (n-n_0)}{k^2}z + \frac{\pi}{12} \frac{c_3(k)}{z} 
\right]\dd{z}.
\end{equation}
Referring to the definition of $\xi(h,k)$ in Eq. \ref{eq:xi_def} and of the 
Kloosterman-like sum $A_k(n)$ in Eq. \ref{eq:klooster_defs}, this is exactly
\begin{equation} \label{eq:finalComFirst}
d(n) = i\sum_{k=1}^\infty k^{-(2+c_1)} c_2(k) A_k(n)\int_{B_{1/2}(1/2)}  
z^{c_1} \exp\left[ \frac{2\pi (n-n_0)}{k^2}z + \frac{\pi}{12} \frac{c_3(k)}{z}
\right]\dd{z},
\end{equation}
Now we just evaluate this integral 
\begin{equation}
I = \int_{B_{1/2}(1/2)}  z^{c_1} \exp\left[ \frac{2\pi 
	(n-n_0)}{k^2}z + 
\frac{\pi}{12} \frac{c_3(k)}{z} \right]\dd{z}.
\end{equation}
First note that if $c_3(k)=0$, then the integrand is everywhere holomorphic so 
that by the Cauchy integral formula $I=0$. Otherwise, we can rewrite it as
\begin{equation}
I=\int_{B_{1/2}(1/2)}  z^{c_1} \exp\left[ \frac{\pi}{k} 
\sqrt{\frac{|c_3(k)|}{6}(n-n_0) }\left( \left( \frac{z}{k} \sqrt{ \frac{24 
(n-n_0)}{ |c_3(k)|} 
} \right)  \pm \left( \frac{z}{k} \sqrt{ \frac{24 
(n-n_0)}{ |c_3(k)|} }  \right) 
^{-1}\right)\right]\dd{z} 
\end{equation}
with the $\pm$ given by $\operatorname{sign}(c_3(k))$.
We make the substitution
\begin{equation}
w = \left( \frac{z}{k} \sqrt{ \frac{24 (n-n_0)}{ |c_3(k)|} }  \right) ^{-1}, 
\quad z = 
\left( \frac{w}{k} \sqrt{ \frac{24 (n-n_0)}{ |c_3(k)|} }  \right) ^{-1}, \quad 
\dd{z}=- 
\left( \frac{w^2}{k} \sqrt{ \frac{24 (n-n_0)}{ |c_3(k)|} }  \right) ^{-1} 
\dd{w},
\end{equation}
whence
\begin{equation}
I = -\left(k\sqrt{ \frac{ |c_3(k)|}{24 (n-n_0)} }  \right)^{c_1+1} 
\int^{1+i\infty}_{1-i\infty}
w^{-(c_1+2)} 
\exp\left[ \frac{\pi}{k} \sqrt{\frac{|c_3(k)|}{6}(n-n_0) }\left(  w^{-1} \pm 
w\right)\right] \dd{w}.
\end{equation}
We now split into two cases depending on the sign of $c_3(k)$. If $c_3(k)<0$, 
then the integrand decays sufficiently rapidly in the right-half plane such 
that 
\begin{equation}
I = -\left(k\sqrt{ \frac{ |c_3(k)|}{24 (n-n_0)} }  \right)^{c_1+1} 
\lim_{R\to\infty}\int_{S(R)}
w^{-(c_1+2)} 
\exp\left[ \frac{\pi}{k} \sqrt{\frac{|c_3(k)|}{6}(n-n_0) }\left(  w^{-1}- 
w\right)\right] \dd{w}
\end{equation}
where $S(R)$ is the right semicircle of radius $R$, centered at $1$, traversed 
clockwise. The 
integrand is holomorphic inside this contour, since it does not contain the 
origin, so that by the Cauchy integral 
formula $I = 0$. Otherwise, if $c_3(k)>0$, then the integrand decays 
sufficiently rapidly in the 
left-half plane such that 
\begin{equation}
I = -\left(k\sqrt{ \frac{| c_3(k)|}{24 (n-n_0)} }  \right)^{c_1+1} 
\oint
w^{-(c_1+2)} 
\exp\left[ \frac{\pi}{k} \sqrt{\frac{|c_3(k)|}{6}(n-n_0) }\left(  w^{-1} + 
w\right)\right] \dd{w}
\end{equation}
for any positively oriented closed contour winding once around the origin.
We rearrange the terms in the integral slightly:
\begin{equation}
I = -2\pi i\left(k\sqrt{ \frac{ |c_3(k)|}{24 (n-n_0)} }  \right)^{c_1+1}
\frac{1}{2\pi i} \oint w^{-(c_1+1)-1} \exp\left[ \frac{1}{2}\frac{\pi}{k} 
\sqrt{\frac{2}{3}|c_3(k)|(n-n_0) }\left(  w^{-1} + w\right)\right] \dd{w}.
\end{equation}
This integral is a standard form of the modified Bessel function of the first 
kind 
\cite{boas2006mathematical}\cite{weisstein2006modified}:
\begin{equation}
I = -2\pi i\left(k\sqrt{ \frac{ |c_3(k)|}{24 (n-n_0)} }  \right)^{c_1+1} 
I_{1+c_1}\left[ \frac{\pi}{k}\sqrt{\frac{2}{3}|c_3(k)| (n-n_0)} \right]
\end{equation}
where $I_{c_1+1}$ is the modified Bessel function of the first kind of weight 
$c_1$. To summarize,
\begin{equation}
I = 
\begin{cases}
0 & (c_3(k)\leq 0), \\ 
-2\pi i\left(k\sqrt{ \frac{ c_3(k)}{24 (n-n_0)} }  \right)^{c_1+1} 
I_{1+c_1}\left[ \frac{\pi}{k}\sqrt{\frac{2}{3}c_3(k)(n-n_0)} \right] & 
(c_3(k)>0).
\end{cases}
\end{equation}
Substituting $I$ into Eq. \ref{eq:finalComFirst} and simplifying, we get our 
final expression
\begin{equation}
d(n) = 2\pi \left( \frac{1}{24 (n-n_0)} 
\right)^{\frac{c_1+1}{2}}\sum_{\substack{k=1 \\ c_3(k)>0}}^\infty  
c_2(k) c_3(k)^{\frac{c_1+1}{2}} k^{-1} A_k(n)  
I_{1+c_1}\left[ 
\frac{\pi}{k}\sqrt{\frac{2}{3}c_3(k) (n-n_0)} \right]. 
\label{eq:proof_final_form}
\end{equation} 
\hfill $\square$

\section{Asymptotics} \label{sec:asymp}

We would like to extract useful asymptotics from Eq. \ref{eq:proof_final_form}. 
These are contained in the following proposition. For this section we assume 
that the hypotheses of Theorem \ref{main} are satisfied, so that Eq. 
\ref{eq:1_rademacher_main} applies.

\begin{proposition}\label{prop:asymp}
	Let $\calK\subset \NN$ be the set of $k$ that maximize $c_3(k)/k^2$ and 
	let $c_3>0$ be the maximum value. For any $\epsilon>0$, there exists some 
	constant $C>0$ which 
	may depend only on $\{\delta_m\}_{m=1}^\infty$ such that for 
	all $n\in\NN$ 
	with $n>n_0$ and 
	\begin{equation}
	\left|\sum_{k\in\calK} c_2(k) k^{c_1} A_k(n)\right| > \epsilon 
	\label{eq:asympt_nonzero_const}
	\end{equation}
	it is the case that 
	\begin{equation}
	d(n) = (1+O(e^{-C\sqrt{n}}))\cdot 2\pi 
	\left(\frac{c_3}{24(n-n_0)}\right)^{\frac{c_1+1}{2}} I_{1+c_1}\left[ \pi 
	\sqrt{\frac{2}{3}c_3 (n-n_0)}\right] \sum_{k\in\calK} c_2(k) k^{c_1} 
	A_k(n). 
	\end{equation}
	The dependence on $\epsilon$ is in the bounding coefficient of the 
	$O(e^{-C\sqrt{n}})$.
	$\hfill\blacksquare$
\end{proposition}
The proof of Prop. \ref{prop:asymp} is straightforward and an exercise in using 
the asymptotics of the 
modified Bessel functions. 
\begin{proof}
	Note that $c_3(k)$ is periodic with period $\lcm(\calM)$. So, 
	$\calK\subseteq \{1,\ldots,\lcm(\calM)\}$. We first break 
	up the Rademacher series in Eq. \ref{eq:1_rademacher_main} into 
	$\lcm(\calM)$ sums, one for each 
	possible value of $k$ modulo $\lcm(\calM)$. We then show that each of 
	these smaller sums is exponentially dominated by the leading term. We then 
	absorb the resulting Bessel functions with $k\not\in \calK$ into those with 
	$k\in\calK$. In the following, we will use the 
	result that if $0<a<b$ then $I(ax)$ is exponentially 
	dominated by $I(bx)$ for any positive $x$ and any positive weight modified 
	Bessel function $I$ of the first kind. 
	So, we first consider for fixed $b\in\{1,\ldots,\lcm(\calM)\}$ with 
	$c_3(b)>0$
	\begin{equation}
	\sum_{\substack{k \in [b]\\ k>0}} c_2(k)c_3(k)^{(c_1+1)/2} 
	A_k(n)k^{-1} I_{1+c_1}\left[\frac{\pi}{k} \sqrt{\frac{2}{3} c_3(k) 
	(n-n_0)}\right]
	\end{equation}
	where $[b]$ is the equivalence class of integers modulo $\lcm(\calM)$.
	Because $c_2(k),c_3(k)$ have period $\lcm(M)$, this is 
	\begin{equation}
	c_2(b)c_3(b)^{(c_1+1)/2}
	\sum_{\substack{k \in [b]\\ k>0}}  A_k(n) 
	k^{-1}I_{1+c_1}\left[\frac{\pi}{k} \sqrt{\frac{2}{3} c_3(b) (n-n_0)}\right].
	\label{allbterms}
	\end{equation}
	We wish to show that the sum above is dominated by the first term. Consider 
	the rest of the terms, 
	\begin{equation}
	\sum_{\substack{k \in [b]\\ k>\lcm(\calM)}} 
	A_k(n) k^{-1}I_{1+c_1}\left[\frac{\pi}{k} \sqrt{\frac{2}{3} c_3(b) 
	(n-n_0)}\right],
	\end{equation}
	which is bounded above in absolute value by
	\begin{equation}
	\sum_{\substack{k \in [b]\\ k>\lcm(\calM)}} 
	\left|I_{1+c_1}\left[\frac{\pi}{k} \sqrt{\frac{2}{3} c_3(b) 
	(n-n_0)}\right]\right|.
	\end{equation}
	Using the expansion of $I_{1+c_1}(z)$ \cite{weisstein2006modified},
	\begin{equation}
	I_{1+c_1}\left[\frac{\pi}{k} \sqrt{\frac{2}{3} c_3(b) 
	(n-n_0)}\right]=\sum_{j=0}^\infty \frac{1}{\Gamma(j+c_1+2) 
	j!}\left(\frac{\pi}{2k}\sqrt{\frac{2}{3} c_3(b) (n-n_0)}\right)^{2j+1+c_1}.
	\end{equation}
	Suppose that $k_0$ is a real number satisfying $0<k_0\leq k$.  
	\begin{align}
	\left|I_{1+c_1}\left[\frac{\pi}{k} \sqrt{\frac{2}{3} c_3(b) 
	(n-n_0)}\right]\right|&\leq \sum_{j=0}^\infty \frac{1}{\Gamma(j+c_1+2) 
	j!}\left(\frac{\pi}{2k_0}\sqrt{\frac{2}{3} c_3(b)
	(n-n_0)}\right)^{2j+1+c_1}\left( \frac{k_0}{k} \right)^{2j+1+c_1} \\
	&\leq \left( \frac{k_0}{k} \right)^{1+c_1} \sum_{j=0}^\infty 
	\frac{1}{\Gamma(j+c_1+2) j!}\left(\frac{\pi}{2k_0}\sqrt{\frac{2}{3} 
	c_3(b) (n-n_0)}\right)^{2j+1+c_1}\\
	&=\left( \frac{k_0}{k} \right)^{1+c_1}I_{1+c_1}\left[\frac{\pi}{k_0} 
	\sqrt{\frac{2}{3} c_3(b) (n-n_0)}\right] .
	\end{align}
	Summing over all relevant $k$,
	\begin{equation}
	\sum_{\substack{k \in [b]\\ k>\lcm(\calM)}} 
	\left|I_{1+c_1}\left[\frac{\pi}{k} \sqrt{\frac{2}{3} c_3(b) 
		(n-n_0)}\right]\right| \leq k_0^{1+c_1} \zeta(1+c_1) I_{1+c_1}\left[ 
		\frac{\pi}{k_0} \sqrt{\frac{2}{3} c_3(b) (n-n_0)} \right].
	\end{equation}
	Therefore, setting $k_0 = b+1/2$ yields
	\begin{equation}
	\sum_{\substack{k \in [b]\\ k>\lcm(\calM)}} 
	\left|I_{1+c_1}\left[\frac{\pi}{k} \sqrt{\frac{2}{3} c_3(b) 
	(n-n_0)}\right]\right|=O(e^{-C_b \sqrt{n}}) I_{1+c_1}\left[ \frac{\pi}{b} 
	\sqrt{\frac{2}{3}c_3(b)(n-n_0)} \right]
	\end{equation}
	for some constant $C_b>0$ depending on $b$. Therefore, the expression in 
	Eq. 
	\ref{allbterms} is
	\begin{equation}
	c_2(b)c_3(b)^{(c_1+1)/2}
	(A_b(n)b^{-1}+O(e^{-C_b\sqrt{n}})) I_{1+c_1}\left[\frac{\pi}{k} 
	\sqrt{\frac{2}{3} c_3(b) (n-n_0)}\right].
	\end{equation}
	It follows that 
	\begin{equation}
	\sum_{\substack{k \in [b]\\ k>0}}  c_2(k)c_3(k)^{(c_1+1)/2} 
	A_k(n)k^{-1} I_{1+c_1}\left[\frac{\pi}{k} \sqrt{\frac{2}{3} c_3(k) 
	(n-n_0)}\right]
	\end{equation}
	is 
	\begin{equation}
	c_2(b)c_3(b)^{(c_1+1)/2}
	(A_b(n)b^{-1}+O(e^{-C_b \sqrt{n}})) I_{1+c_1}\left[\frac{\pi}{b} 
	\sqrt{\frac{2}{3} c_3(b) (n-n_0)}\right].
	\end{equation}
	We can sum this result for all $b\in\{1,\ldots,\lcm(\calM)\}$ with 
	$c_3(b)>0$. We can absorb the terms with $b\not\in \calK$ into the error 
	term.
	So, for some constant $C$ depending only on $\{\delta_m\}_{m=1}^\infty$,
	\begin{equation}
	d(n) = 2\pi \left(\frac{1}{24(n-n_0)}\right)^{\frac{c_1+1}{2}}\sum_{k\in 
	\calK}(A_k(n)k^{-1}+O(e^{-C\sqrt{n}}))c_2(k) c_3(k)^{\frac{c_1+1}{2}} 
	I_{1+c_1}\left[ \frac{\pi}{k}\sqrt{\frac{2}{3}c_3(k) (n-n_0)} \right].
	\end{equation}
	Since for $k\in\calK$ it is the case that $c_3(k)=k^2c_3$, 
	\begin{equation}
	d(n) =2\pi  \left(\frac{c_3}{24(n-n_0)}\right)^{\frac{c_1+1}{2}} 
	I_{1+c_1}\left[ 
	\pi\sqrt{\frac{2}{3}c_3 (n-n_0)} \right] \sum_{k\in 
	\calK}(A_k(n)k^{-1}+O(e^{-C\sqrt{n}}))c_2(k)k^{1+c_1} .
	\label{eq:second_to_last_asymp}
	\end{equation}
	Note that each $O(e^{-C\sqrt{n}})$ in the sum over $k$ in Eq. 
	\ref{eq:second_to_last_asymp} is different. Nevertheless, using the 
	assumption in Eq. \ref{eq:asympt_nonzero_const},
	\begin{equation}
	d(n) = (1+O(e^{-C\sqrt{n}})) 2\pi 
	\left(\frac{c_3}{24(n-n_0)}\right)^{\frac{c_1+1}{2}}  
	I_{1+c_1}\left[ \pi\sqrt{\frac{2}{3}c_3 (n-n_0)} \right]\sum_{k\in 
	\calK}A_k(n)k^{c_1}c_2(k) ,
	\label{eq:final_approx}
	\end{equation}
	as claimed, where now the $O(e^{-C\sqrt{n}})$ term is bounded in terms of 
	$\{\delta_m\}_{m=1}^\infty$ and $\epsilon$.
\end{proof}

\section{Numerics} \label{sec:numerics}

In this section we numerically test Eq. 
\ref{eq:1_rademacher_main} 
for several $\eta$-quotients $Z(q)$. Here $d(n,N)$ represents the $N$th partial 
sum of the right hand side of Eq. \ref{eq:1_rademacher_main} and $d(n)$ 
represents the Fourier coefficients of $Z(q)\cdot q^{n_0}$. The following 
$\eta$-quotients all satisfy the hypotheses of Theorem 
\ref{main}. The absolute difference between $d(n,N)$ and $d(n)$ for 
$n\in\{1,\ldots,20\}$ with $n>n_0$ and $N\in\{1,\ldots,100\}$ is plotted in 
Fig. \ref{fig:plotGood} for 
\begin{enumerate}
	\item $Z(q) = 1/\eta(4\tau)\eta(\tau)^3$ (upper left),
	\item $Z(q) = \eta(4\tau)/\eta(\tau)^3$ (upper right),
	\item $Z(q) = 1/\eta(2\tau)$ (middle left),
	\item $Z(q) = 1/\eta(11\tau)^2 \eta(\tau)^2$ (middle right),
	\item $Z(q) = 1/\eta(\tau)\eta(22\tau)$ (bottom left), 
	\item $Z(q) = 1/\eta(\tau)\eta(23\tau)$ (bottom right).
\end{enumerate}
The convergence of $d(n,N)$ to $d(n)$ as $N\to\infty$ is 
clear, although a few trends are worth noting. The first is that the 
convergence of $d(n,N)$ to $d(n)$ is rather haphazard. The second is that the 
rate of convergence of 
$d(n,N)$ to $d(n)$ may depend significantly on $n$. This appears to be the case 
for the final 
few $\eta$-quotients above.

\begin{figure}[t]
		\centering
		\includegraphics[width=.95\textwidth]{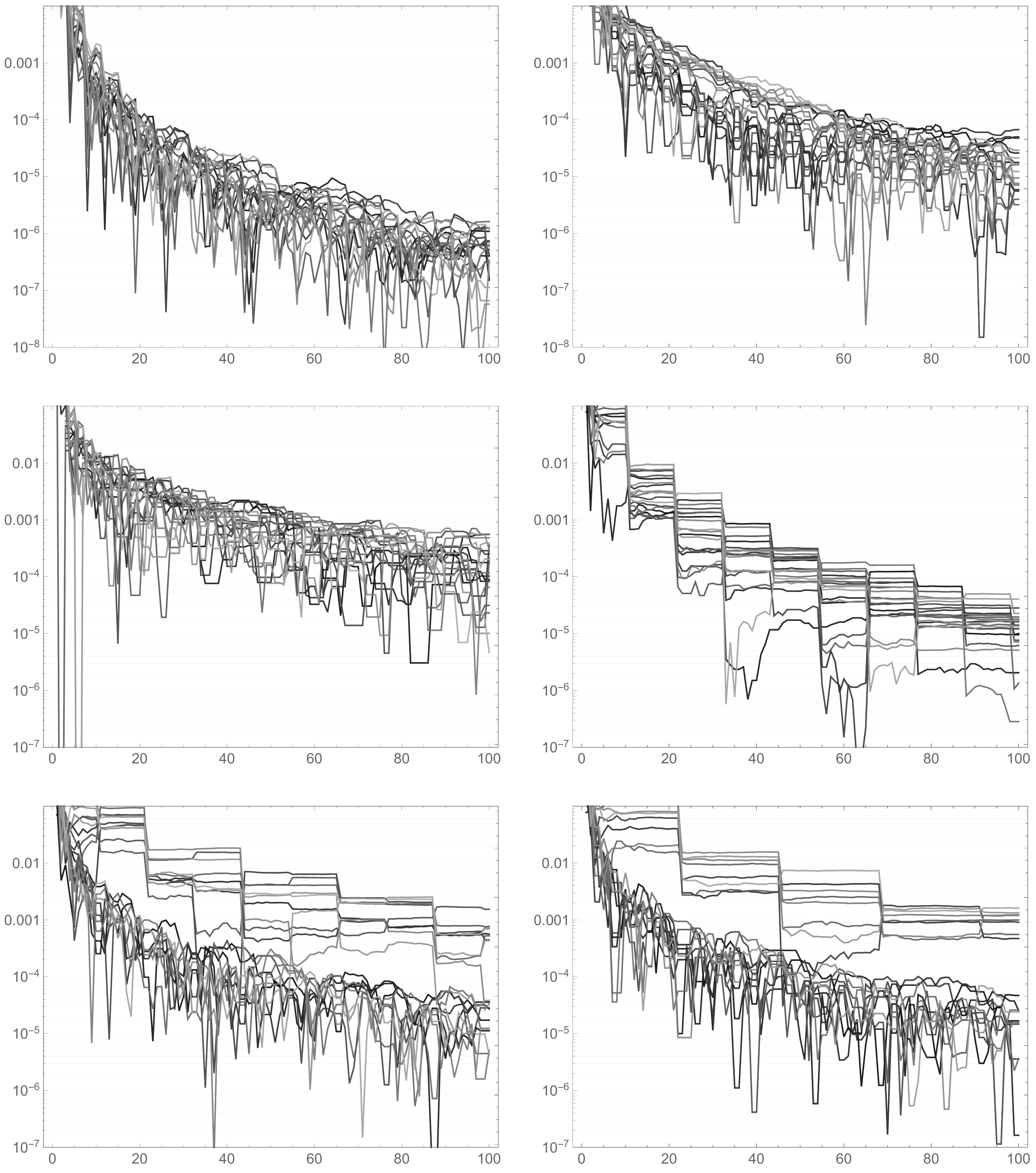}
		\caption{Convergence of the $N$th partial sums of Eq. 
			\ref{eq:1_rademacher_main} for the listed $\eta$-quotients: $Z(q) = 
			1/\eta(4\tau)\eta(\tau)^3$ (upper left),
			$Z(q) = \eta(4\tau)/\eta(\tau)^3$ (upper right),
			$Z(q) = 1/\eta(2\tau)$ (middle left),
			$Z(q) = 1/\eta(11\tau)^2 \eta(\tau)^2$ (middle right),
			$Z(q) = 1/\eta(\tau)\eta(22\tau)$ (bottom left), 
			$Z(q) = 1/\eta(\tau)\eta(23\tau)$ (bottom right). The vertical axis 
			is $|d(n,N)-d(n)|$ and the horizontal axis is $N$. The vertical 
			axis is scaled logarithmically and the horizontal axis is scaled 
			linearly. Each line is a plot of $d(n,N)$ for 
			fixed $n$ and variable $N$. The lines are different shades of gray 
			to help visually distinguish them.}
		\label{fig:plotGood}
\end{figure}

\clearpage
\section*{Acknowledgements}

I would like to thank Prof. Shamit Kachru not only for introducing me to 
the topic of small black holes, but also for taking the time to patiently and 
kindly 
answer my many 
questions. I am also grateful to the other students in SITP for making it such 
a welcoming place, and especially Richard Nally and Brandon Rayhaun for their 
helpful conversations and general advice. Finally, I would like to express my 
gratitude towards Prof. John Duncan for pointing out an error in an earlier 
version 
of this paper.
While working on this  
project I was funded by the Stanford physics department summer research program 
and a Stanford UAR major grant.

\bibliographystyle{unsrt}
\bibliography{RademacherBib}

\end{document}